\documentclass[12pt]{article}
\usepackage {graphics}
\usepackage{comment}
\usepackage {pstricks,pst-node}
\usepackage{amsmath,amssymb,amsfonts,oldlfont,latexsym,dsfont,amsthm}
\usepackage[T1]{fontenc}

\usepackage{float}

\onecolumn
\begin{document}

\newtheorem{thm}{Theorem}
\newtheorem{prop}[thm]{Proposition}
\newtheorem{conj}[thm]{Conjecture}
\newtheorem{ob}{Observation}

\newcommand{\trou}{\vspace{5 mm} \noindent}
\newcommand{\Trou}{\vspace{7 mm} \noindent}
\newcommand{\tour}{\vspace{4 mm}}
\newcommand{\ptrou}{\vspace{4 mm} \noindent}
\newcommand{\cbdo}{\hfill \rule{.1in}{.1in}}    
\newcommand{\prf}{\noindent{ \bf Proof.}\ }
\newcommand{\cex}{\noindent{\bf Counterexample.\ }}

\newcommand{\n}{I\!\!N}
\newcommand{\z}{{\sf Z}\!\!{\sf Z}}
\newcommand{\q}{Q\!\!\!\!I}
\newcommand {\rn}[1]{{I\!\!R}^{#1}}     
\newcommand{\ce}{C\!\!\!\!I}        
\newcommand {\cen}[1]{{C\!\!\!\!I}\,^{#1}}        
\newcommand{\pusty}{\mbox{\sc\O}}


\title{A note \\on sequences variant \\ 
of  irregularity strength for hypercubes}
\author{Anna Flaszczy\'nska, Aleksandra Gorzkowska \\and Mariusz Wo\'zniak \\
Department of Discrete Mathematics\\
AGH University}

\maketitle
\begin{abstract}
Let $f: E \mapsto \{1,2,\dots,k\}$ be an edge coloring of the $n$ - dimensional hypercube $H_n$. By the palette at a vertex $v$ we mean the sequence $\left(f(e_1(v)), f(e_1(v)),\dots, f(e_n(v))\right)$, where $e_i(v)$ is the $i$ - dimensional edge incident to $v$. In the paper, we show that two colors are enough to distinguish all vertices of the $n$ - dimensional hypercube $H_n$ ($n \geq 2$) by their palettes. We also show that if $f$ is a proper edge coloring of the hypercube $H_n$ ($n\geq 5$), then $n$ colors suffice to distinguish all vertices by their palettes.
\end{abstract}

\section{Introduction}
\par

The problem of distinguishing the vertices of a graph by
edge coloring has attracted the attention of many authors, resulting in a variety of papers being written in this area. These papers are all based on the concept of a color palette at a vertex,
which is a set (or a multiset) of colors present at a given vertex of the graph.

The first work on this subject was probably a paper by Harary
and Plantholt \cite{HaPl85} from 1985 where they consider the problem of distinguishing all vertices  by sets, and the edge coloring in question is general 
(\emph{i.e.} not necessarily proper).  
The variant of this problem concerning distinguishing vertices by 
multisets appeared in 1990 in the paper of Aigner and Triesch \cite{AiTr90a}.

The idea to consider proper edge colorings appeared independently in two places. In Ko{\v s}ice, in R. Soták's master thesis, and in Memphis, in the doctoral thesis of A. Burris. Respective publications were published with a certain delay (cf. \cite{CHS96}, \cite{BuSc97}).

The next natural variant in this area is to use a sequence of colors as the palette at a vertex. This idea comes from the paper of Seamone and Stevens \cite{SeSt13}. In this paper, the order of edges incident to a given vertex is induced by the global order of edges,
which is a given bijection of the set $E$ and the set $\{1,2, 3, \ldots, |E|\}$.
In our case, the order of the edges follows {\bf naturally} from the concept of the dimension of an edge in the hypercube.

We begin by stating the basic definitions. The \emph{$n$-dimensional hypercube} $H_n$ is a graph whose vertices are binary sequences of $n$ elements, and whose edges connect two vertices if corresponding sequences differ in exactly one element. If the differing element is the $i$th element of the sequences, we say that the edge is of \emph{dimension} $i$.

We will often use the recursive definition of a hypercube according to which $H_0=K_1$, and an $n$-dimensional hypercube is created from two $(n-1)$-dimensional hypercubes, $H_{n-1}$ and $H'_{n-1}$, by adding a matching 
between corresponding vertices in the hypercubes. This matching consists of edges of dimension $n$.

Let us denote by $E_i$ the edges of dimension $i$ of the hypercube, and for $x\in V(H_n)$ denote by $e_i(x)$ the edge of dimension $i$ incident to the vertex $x$. Note that the concept of edge dimension gives the set of edges incident to $x$ a {\bf natural sequence} structure. Let $H_n=(V,E)$ be the $n$-dimensional hypercube, let $f: E\mapsto \{1,2,\ldots, k\}$ be a coloring of its edges. Let $x$ be a vertex of the hypercube. By the \emph{palette at vertex} $x$ we mean the sequence $F(x) = (f(e_1(x)), f(e_2(x)), \ldots, f(e_n(x)))$.
We say that two vertices \emph{are distinguished by} $f$ if their palettes are different. In this paper, we find the smallest number of colors necessary to distinguish all vertices of the hypercube. Clearly, this is only possible for hypercubes of dimension $n \geq 2$.

\section{General coloring}

In this section, we consider general edge colorings. We note that the hypercube $H_n$ has $2^n$ vertices. Naturally, we need at least as many different sequences to distinguish all vertices. Using $k$ colors in the edge coloring, there are $k^n$ possible sequences. In the case of a general coloring, this leads to a trivial bound $k \geq 2$. We show that, indeed, two colors suffice to distinguish all vertices by sequences in a general coloring.

\begin{thm} \label{thm:general_2}
The smallest number of colors needed to distinguish the vertices of a hypercube $H_n$, $n \geq 2$, by general coloring is two.
\end{thm}

\begin{proof}
We use $\bar a$ to denote $0$ when $a$ is $1$ and $1$ when $a$ is $0$.
We proceed by induction on the dimension $n\geq  2$ of a hypercube $H_n$. 
If $n=2$, then the coloring that distinguishes the vertices is shown in Figure \ref{gen}.

Suppose that $n\geq 3$ and, there is a coloring $f$ of the hypercube $H_{n-1}$ with two colors ($0$ and $1$) such that all vertices are distinguished by $f$. Let us consider the hypercube $H_n$ as two copies of the $n-1$-dimensional hypercube, denoted by $H_{n-1}$ and $H_{n-1}'$, connected by a perfect matching. To create a coloring $f'$ of the edges of $H_n$, we use the coloring $f$ on the edges of $H_{n-1}$ and $H_{n-1}'$. Next, we interchange the colors $0$ and $1$ on the edges of $H_{n-1}'$. 

Let us take a vertex $v$ of the graph $H_{n-1}$, denote the first $n-1$ elements of its palette by $(a_1,a_2,\dots, a_{n-1})$. We color the edge between $v$ and its neighbor $w$ in the graph $H_{n-1}'$ with color $0$. Then, we consider the vertex from the hypercube $H_{n-1}$ with the first $n-1$ elements of the palette of the form $(\bar a_1, \bar a_2, \dots, \bar a_{n-1})$ and color the edge connecting this vertex with the corresponding vertex in $H_{n-1}'$ with color $1$. We then consider the next vertex from $H_{n-1}$ which still has an incident edge uncolored, and proceed as before until all edges are colored.

Note that the coloring $f'$ distinguishes all vertices of the hypercube $H_n$. This follows directly from the assumption about the coloring $f$ and the fact that for all pairs of vertices $u\in H_{n-1}$, $v \in H_{n-1}'$ whose first $n-1$ elements of the pallets are the same, the edge in the $n$th dimension has a different color.
\end{proof}

\begin{figure} 
\psset{unit=1cm}
\psset{radius=0.2}

\begin{pspicture}(12,3)

\put(0.5,0.5)
{\begin{pspicture}(0,0)
\dotnode(0,0){X1} 
\dotnode(2,0){X2} 
\dotnode(2,2){X3}
\dotnode(0,2){X4}
\ncline{X1}{X2}
\ncline{X2}{X3}
\ncline{X3}{X4}
\ncline{X4}{X1}

\put(-0.5,0){$00$}
\put(2.1,0){$01$}
\put(2.1,1.7){$11$}
\put(-0.5,1.7){$10$}

\rput(1,0){\circlenode[linestyle=none,fillstyle=solid,fillcolor=white]{C1}{\mbox{\small 0}}}
\rput(2,1){\circlenode[linestyle=none,fillstyle=solid,fillcolor=white]{C1}{\mbox{\small 1}}}
\rput(1,2){\circlenode[linestyle=none,fillstyle=solid,fillcolor=white]{C1}{\mbox{\small 1}}}
\rput(0,1){\circlenode[linestyle=none,fillstyle=solid,fillcolor=white]{C1}{\mbox{\small 0}}}
\end{pspicture}}

\put(6.5,0.5)
{\begin{pspicture}(0,0)
\dotnode(0,0){X1} 
\dotnode(2,0){X2} 
\dotnode(2,2){X3}
\dotnode(0,2){X4}
\ncline{X1}{X2}
\ncline{X2}{X3}
\ncline{X3}{X4}
\ncline{X4}{X1}

\rput(1,0){\circlenode[linestyle=none,fillstyle=solid,fillcolor=white]{C1}{\mbox{\small 0}}}
\rput(2,1){\circlenode[linestyle=none,fillstyle=solid,fillcolor=white]{C1}{\mbox{\small 1}}}
\rput(1,2){\circlenode[linestyle=none,fillstyle=solid,fillcolor=white]{C1}{\mbox{\small 1}}}
\rput(0,1){\circlenode[linestyle=none,fillstyle=solid,fillcolor=white]{C1}{\mbox{\small 0}}}

\dotnode(4,0){Y1} 
\dotnode(6,0){Y2} 
\dotnode(6,2){Y3}
\dotnode(4,2){Y4}
\ncline{Y1}{Y2}
\ncline{Y2}{Y3}
\ncline{Y3}{Y4}
\ncline{Y4}{Y1}

\rput(5,0){\circlenode[linestyle=none,fillstyle=solid,fillcolor=white]{C1}{\mbox{\small 1}}}
\rput(6,1){\circlenode[linestyle=none,fillstyle=solid,fillcolor=white]{C1}{\mbox{\small 0}}}
\rput(5,2){\circlenode[linestyle=none,fillstyle=solid,fillcolor=white]{C1}{\mbox{\small 0}}}
\rput(4,1){\circlenode[linestyle=none,fillstyle=solid,fillcolor=white]{C1}{\mbox{\small 1}}}

\pscurve(0,0)(2,-0.5)(4,0)
\pscurve(2,0)(4,-0.5)(6,0)
\pscurve(0,2)(2,2.5)(4,2)
\pscurve(2,2)(4,2.5)(6,2)

\rput(2,-0.5){\circlenode[linestyle=none,fillstyle=solid,fillcolor=white]{C1}{\mbox{\small 0}}}
\rput(4,-0.5){\circlenode[linestyle=none,fillstyle=solid,fillcolor=white]{C1}{\mbox{\small 0}}}
\rput(2,2.5){\circlenode[linestyle=none,fillstyle=solid,fillcolor=white]{C1}{\mbox{\small 1}}}
\rput(4,2.5){\circlenode[linestyle=none,fillstyle=solid,fillcolor=white]{C1}{\mbox{\small 1}}}

\end{pspicture}}

\end{pspicture}
\caption {Edge coloring of $H_2$ and $H_3$ distinguishing all vertices} \label{gen}
\end{figure}
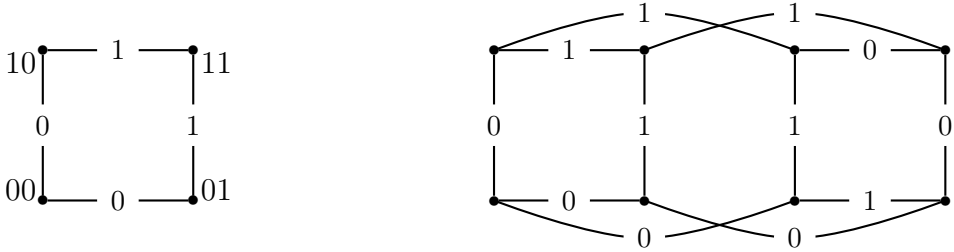

\section{Proper coloring}

In this section, we turn our attention to proper colorings. It is a well-known fact that the chromatic index of a hypercube $H_n$ is $n$. Therefore, a trivial natural lower bound for our problem is $n$ colors. We show that if the dimension of a hypercube is at least five, $n$ colors suffice. However, if $n \leq 4$ more colors are needed. We consider these cases separately in the next subsections.
 
\subsection{Cases for $n\le 4$}

\begin{prop}
The smallest number of colors needed to distinguish the vertices of hypercubes $H_2$ and $H_3$ by proper edge coloring is four.
\end {prop}

\begin{proof}
It is easy to verify that the proposition holds in the case of the hypercube $H_2$. For the hypercube $H_3$, at least four colors are needed. It follows from the fact that the number of vertices in $H_3$ is greater than the number of permutations of $3$ elements. Figure \ref{H3} presents the distinguishing coloring of $H_3$ with four colors.
\end{proof}

\begin{figure}[H]
\psset{unit=0.5cm}
\psset{radius=0.2}

\begin{center}

\begin{pspicture}(6,7)

\put(0.5,0.5)
{\begin{pspicture}(0,0)
\dotnode(0,0){X1} 
\dotnode(5,0){X2} 
\dotnode(5,5){X3}
\dotnode(0,5){X4}
\ncline{X1}{X2}
\ncline{X2}{X3}
\ncline{X3}{X4}
\ncline{X4}{X1}

\dotnode(1.5,1.5){Y1} 
\dotnode(3.5,1.5){Y2} 
\dotnode(3.5,3.5){Y3}
\dotnode(1.5,3.5){Y4}
\ncline{Y1}{Y2}
\ncline{Y2}{Y3}
\ncline{Y3}{Y4}
\ncline{Y4}{Y1}

\ncline{X1}{Y1}
\ncline{X2}{Y2}
\ncline{X3}{Y3}
\ncline{X4}{Y4}

\rput(2.5,0){\circlenode[linestyle=none,fillstyle=solid,fillcolor=white]{C1}{\mbox{\small 1}}}
\rput(2.5,5){\circlenode[linestyle=none,fillstyle=solid,fillcolor=white]{C1}{\mbox{\small 1}}}
\rput(0.75,0.75){\circlenode[linestyle=none,fillstyle=solid,fillcolor=white]{C1}{\mbox{\small 4}}}
\rput(4.25,0.75){\circlenode[linestyle=none,fillstyle=solid,fillcolor=white]{C1}{\mbox{\small 4}}}
\rput(5,2.5){\circlenode[linestyle=none,fillstyle=solid,fillcolor=white]{C1}{\mbox{\small 2}}}
\rput(0,2.5){\circlenode[linestyle=none,fillstyle=solid,fillcolor=white]{C1}{\mbox{\small 3}}}
\rput(1.5,2.5){\circlenode[linestyle=none,fillstyle=solid,fillcolor=white]{C1}{\mbox{\small 1}}}

\rput(2.5,1.5){\circlenode[linestyle=none,fillstyle=solid,fillcolor=white]{C1}{\mbox{\small 3}}}
\rput(0.75,4.25){\circlenode[linestyle=none,fillstyle=solid,fillcolor=white]{C1}{\mbox{\small 2}}}
\rput(4.25,4.25){\circlenode[linestyle=none,fillstyle=solid,fillcolor=white]{C1}{\mbox{\small 3}}}
\rput(2.5,3.5){\circlenode[linestyle=none,fillstyle=solid,fillcolor=white]{C1}{\mbox{\small 4}}}
\rput(3.5,2.5){\circlenode[linestyle=none,fillstyle=solid,fillcolor=white]{C1}{\mbox{\small 2}}}

\end{pspicture}}

\end{pspicture}
\end{center}

\caption {Proper edge coloring of $H_3$ distinguishing all vertices} 
\label{H3}
\end{figure}
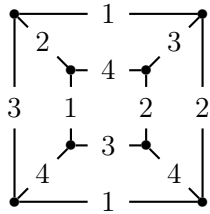

\begin{prop}
Four colors are not enough to distinguish the vertices of a hypercube $H_4$ by proper edge coloring.
\end{prop}

\begin{proof} 
We will consider the hypercube $H_4$ as two three-dimensional hypercubes, $H_3$ and $H'_3$, with a matching connecting the corresponding vertices. The edges of this matching are of the 4th dimension. The vertices of the hypercube  $H_3$ are denoted by $x_i$ and $y_i$, and the vertices of the  hypercube $H'_3$ 
are denoted by $x'_i$ and $y'_i$ (see Figure \ref{prop4_case1A1}).
We follow the convention that the horizontal edges in both hypercubes are of dimension $1$, the vertical edges of dimension $2$, and the diagonal edges of dimension $3$.

Suppose that there is a proper coloring $f$ of the edges of a hypercube $H_4$ with four colors, such that all vertices are distinguished by sequences $F(x)$. Note that in such a coloring, every vertex has all four colors in its palette.

Let us first assume that at least four edges of some dimension have the same color. Therefore, to distinguish the eight vertices that are the ends of these four edges, we have only three palette coordinates at each vertex and three colors. Therefore, we have six possibilities for a palette and eight vertices to distinguish, which is impossible. Hence, at most three edges of one dimension have the same color.

Now we assume that exactly three edges of some dimension, say the 4th, have the same color, say 1. Since every vertex has all four colors in its palette, the edges in color $1$ form a perfect matching in $H_4$. However, there are eight vertices in $H_3$, and for exactly three of them the edge of that matching is in the 4th dimension. Therefore, it is impossible to complete the perfect matching in $H_3$ with five vertices, a contradiction.

Suppose that there is a color $k$ and a dimension $i$ such that exactly one edge of the $i$th dimension has color $k$. Since there are eight edges of dimension $i$, there are at least three edges of dimension $i$ with a color different from $k$, a contradiction. Therefore, we assume that for each dimension, there are exactly two edges of each color.
We will consider cases depending on the distance of edges having the same color. We say that two edges of the same dimension $e=xy$ and $e'=x'y'$ are at a distance $d$ if
$\min \{{\rm dist}(x,x'), {\rm dist}(x,y')\}=d$.

In the proof, we will use three rules:
\begin{description}
	\item[R1] coloring $f$ is proper,
	\item[R2] there are no three edges of the same color in one dimension, 
    \item[R3] all vertices are distinguished by sequences.
\end{description}

$\newline$
\noindent{\textbf{Case 1}} There are two edges of the same dimension and the same color at distance $1$.

\trou

Without loss of generality, we assume that $f(x_1x_2)=f(x_3x_4)=1$ and next that $f(x_2x_3)=2$, by R1.
Applying rules R1 and R3, we color edges $x_2y_2$ and $x_3y_3$ without loss of generality with $3$ and $4$ colors, respectively. It is easy to see that
the edge $x_1x_4$ cannot be colored with colors neither $1$ nor $2$, therefore due to symmetry, we assume that $f(x_1x_4)=3$. Then by R1 edge $x_1 y_1$ can be colored only with color $2$ or $4$. We consider the cases separately.

\trou
\noindent{\textbf{Subcase 1A}} $f(x_1y_1)=2$. 

The edge $y_1 y_2$ has color $4$, because of $R1$ and $R2$. Then the edge $x_4y_4$ has color $4$ (R1 and R3). Applying R1 and R2, we can only use colors $2$ or $3$ to color the edge $y_3y_4$, assume that $f(y_3y_4)=2$ (the other possibility is symmetric). Next $f(y_2y_3)=1$, by R1 and $f(y_1y_4)=3$, by R1 and R3. Due to R1 we obtain that $f(x_3x'_3)=3$, $f(x_4 x'_4)=2$, $f(x_1x'_1)=4$ and $f(y_2y'_2)=2$. From rules R1 and R2 results that $f(x'_3x'_4)=4$ and $f(x'_1x'_4)=1$. Edge $x'_4 y'_4$ has color 3, by R1. The vertices $y_2$ and $x'_4$ received the same sequences, a contradiction.

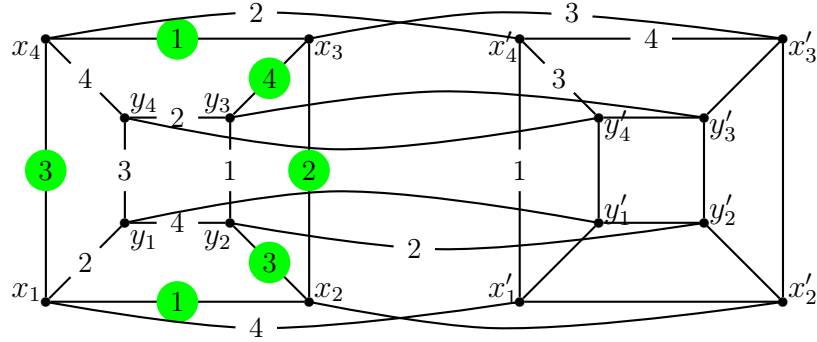
\begin{figure}[H]
\psset{unit=0.7cm}
\psset{radius=0.2}

\begin{center}
\begin{pspicture}(13,6)

\put(0.5,0.5)
{\begin{pspicture}(0,0)
\dotnode(0,0){X1} 
\dotnode(5,0){X2} 
\dotnode(5,5){X3}
\dotnode(0,5){X4}
\ncline{X1}{X2}
\ncline{X2}{X3}
\ncline{X3}{X4}
\ncline{X4}{X1}

\dotnode(1.5,1.5){Y1} 
\dotnode(3.5,1.5){Y2} 
\dotnode(3.5,3.5){Y3}
\dotnode(1.5,3.5){Y4}
\ncline{Y1}{Y2}
\ncline{Y2}{Y3}
\ncline{Y3}{Y4}
\ncline{Y4}{Y1}

\ncline{X1}{Y1}
\ncline{X2}{Y2}
\ncline{X3}{Y3}
\ncline{X4}{Y4}

\put(-0.65,0.1){$x_1$}
\put(5.1,0.1){$x_2$}
\put(-0.65,4.7){$x_4$}
\put(5.1,4.7){$x_3$}

\put(1.6,1.1){$y_1$}
\put(3,1.1){$y_2$}
\put(3,3.7){$y_3$}
\put(1.6,3.7){$y_4$}

\rput(2.5,0){\circlenode[linestyle=none,fillstyle=solid,fillcolor=green]{C1}{\mbox{\small 1}}}
\rput(2.5,5){\circlenode[linestyle=none,fillstyle=solid,fillcolor=green]{C1}{\mbox{\small 1}}}
\rput(0.75,0.75){\circlenode[linestyle=none,fillstyle=solid,fillcolor=white]{C1}{\mbox{\small 2}}}
\rput(4.25,0.75){\circlenode[linestyle=none,fillstyle=solid,fillcolor=green]{C1}{\mbox{\small 3}}}
\rput(5,2.5){\circlenode[linestyle=none,fillstyle=solid,fillcolor=green]{C1}{\mbox{\small 2}}}
\rput(0,2.5){\circlenode[linestyle=none,fillstyle=solid,fillcolor=green]{C1}{\mbox{\small 3}}}
\rput(1.5,2.5){\circlenode[linestyle=none,fillstyle=solid,fillcolor=white]{C1}{\mbox{\small 3}}}

\rput(2.5,1.5){\circlenode[linestyle=none,fillstyle=solid,fillcolor=white]{C1}{\mbox{\small 4}}}
\rput(0.75,4.25){\circlenode[linestyle=none,fillstyle=solid,fillcolor=white]{C1}{\mbox{\small 4}}}
\rput(4.25,4.25){\circlenode[linestyle=none,fillstyle=solid,fillcolor=green]{C1}{\mbox{\small 4}}}
\rput(2.5,3.5){\circlenode[linestyle=none,fillstyle=solid,fillcolor=white]{C1}{\mbox{\small 2}}}
\rput(3.5,2.5){\circlenode[linestyle=none,fillstyle=solid,fillcolor=white]{C1}{\mbox{\small 1}}}

\end{pspicture}}

\put(9.5,0.5)
{\begin{pspicture}(0,0)
\dotnode(0,0){X1} 
\dotnode(5,0){X2} 
\dotnode(5,5){X3}
\dotnode(0,5){X4}
\ncline{X1}{X2}
\ncline{X2}{X3}
\ncline{X3}{X4}
\ncline{X4}{X1}

\dotnode(1.5,1.5){Y1} 
\dotnode(3.5,1.5){Y2} 
\dotnode(3.5,3.5){Y3}
\dotnode(1.5,3.5){Y4}
\ncline{Y1}{Y2}
\ncline{Y2}{Y3}
\ncline{Y3}{Y4}
\ncline{Y4}{Y1}

\ncline{X1}{Y1}
\ncline{X2}{Y2}
\ncline{X3}{Y3}
\ncline{X4}{Y4}

\put(-0.6,0.1){$x'_1$}
\put(5.1,0.1){$x'_2$}
\put(-0.6 ,4.7){$x'_4$}
\put(5.1,4.7){$x'_3$}

\put(1.6,1.6){$y'_1$}
\put(3.6,1.6){$y'_2$}
\put(3.6,3.2){$y'_3$}
\put(1.6,3.2){$y'_4$}

\rput(0,2.5){\circlenode[linestyle=none,fillstyle=solid,fillcolor=white]{C1}{\mbox{\small 1}}}
\rput(0.75,4.25){\circlenode[linestyle=none,fillstyle=solid,fillcolor=white]{C1}{\mbox{\small 3}}}
\rput(2.5,5){\circlenode[linestyle=none,fillstyle=solid,fillcolor=white]{C1}{\mbox{\small 4}}}

\end{pspicture}}

\pscurve(0.5,5.5)(4.5,6)(9.5,5.5)
\pscurve(5.5,5.5)(9,6)(14.5,5.5)
\pscurve(0.5,0.5)(4.5,0)(9.5,0.5)
\pscurve(5.5,0.5)(9,0)(14.5,0.5)

\pscurve(2,4)(6,3.4)(11,4)
\pscurve(4,4)(8,4.5)(13,4)
\pscurve(2,2)(6,2.6)(11,2)
\pscurve(4,2)(8,1.5)(13,2)

\rput(4.5,0){\circlenode[linestyle=none,fillstyle=solid,fillcolor=white]{C1}{\mbox{\small 4}}}
\rput(4.5,6){\circlenode[linestyle=none,fillstyle=solid,fillcolor=white]{C1}{\mbox{\small 2}}}
\rput(10.5,6){\circlenode[linestyle=none,fillstyle=solid,fillcolor=white]{C1}{\mbox{\small 3}}}
\rput(7.5,1.5){\circlenode[linestyle=none,fillstyle=solid,fillcolor=white]{C1}{\mbox{\small 2}}}

\end{pspicture}
\end{center}
\caption {Subcase 1A}   
\label{prop4_case1A1}
\end{figure}

\trou
\noindent{\textbf{Subcase 1B}} $f(x_1y_1)=4$. 
\newline
Due to R1 and R2, the edge $y_1 y_2$ has color $2$. Applying R1, we get that $f(y_2 y_3)=1$. Next $f(y_2y'_2)=4$ and $f(x_2x'_2)=4$, because of R1. Hence, $f(y_4 y'_4) \ne 4$, due to R2. Moreover, edges $y_3 y_4$ and $y_1 y_4$ cannot be colored with color $4$, because of R1. Hence, $f(x_4y_4)=4$, which contradicts R2.

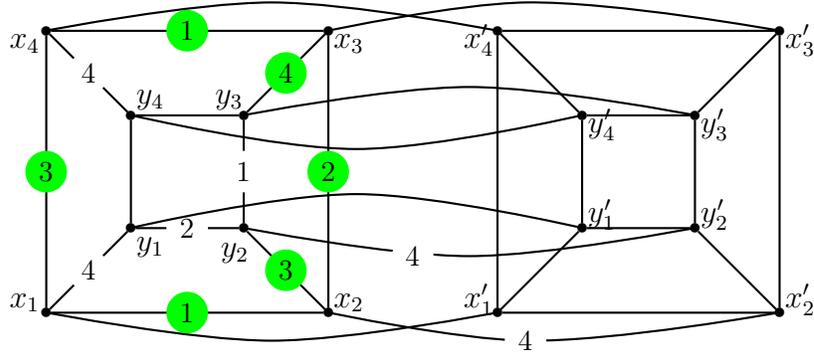
\begin{figure}
\psset{unit=0.75cm}
\psset{radius=0.2}

\begin{center}

\begin{pspicture}(13,6)

\put(0.5,0.5)
{\begin{pspicture}(0,0)
\dotnode(0,0){X1} 
\dotnode(5,0){X2} 
\dotnode(5,5){X3}
\dotnode(0,5){X4}
\ncline{X1}{X2}
\ncline{X2}{X3}
\ncline{X3}{X4}
\ncline{X4}{X1}

\dotnode(1.5,1.5){Y1} 
\dotnode(3.5,1.5){Y2} 
\dotnode(3.5,3.5){Y3}
\dotnode(1.5,3.5){Y4}
\ncline{Y1}{Y2}
\ncline{Y2}{Y3}
\ncline{Y3}{Y4}
\ncline{Y4}{Y1}

\ncline{X1}{Y1}
\ncline{X2}{Y2}
\ncline{X3}{Y3}
\ncline{X4}{Y4}

\put(-0.65,0.1){$x_1$}
\put(5.1,0.1){$x_2$}
\put(-0.65,4.7){$x_4$}
\put(5.1,4.7){$x_3$}

\put(1.6,1.1){$y_1$}
\put(3.1,1){$y_2$}
\put(3,3.7){$y_3$}
\put(1.6,3.7){$y_4$}

\rput(2.5,5){\circlenode[linestyle=none,fillstyle=solid,fillcolor=green]{C1}{\mbox{\small 1}}}
\rput(0.75,0.75){\circlenode[linestyle=none,fillstyle=solid,fillcolor=white]{C1}{\mbox{\small 4}}}
\rput(4.25,0.75){\circlenode[linestyle=none,fillstyle=solid,fillcolor=green]{C1}{\mbox{\small 3}}}
\rput(5,2.5){\circlenode[linestyle=none,fillstyle=solid,fillcolor=green]{C1}{\mbox{\small 2}}}
\rput(0,2.5){\circlenode[linestyle=none,fillstyle=solid,fillcolor=green]{C1}{\mbox{\small 3}}}

\rput(2.5,1.5){\circlenode[linestyle=none,fillstyle=solid,fillcolor=white]{C1}{\mbox{\small 2}}}
\rput(0.75,4.25){\circlenode[linestyle=none,fillstyle=solid,fillcolor=white]{C1}{\mbox{\small 4}}}
\rput(4.25,4.25){\circlenode[linestyle=none,fillstyle=solid,fillcolor=green]{C1}{\mbox{\small 4}}}
\rput(2.5,0){\circlenode[linestyle=none,fillstyle=solid,fillcolor=green]{C1}{\mbox{\small 1}}}
\rput(3.5,2.5){\circlenode[linestyle=none,fillstyle=solid,fillcolor=white]{C1}{\mbox{\small 1}}}

\end{pspicture}}

\put(8.5,0.5)
{\begin{pspicture}(0,0)
\dotnode(0,0){X1} 
\dotnode(5,0){X2} 
\dotnode(5,5){X3}
\dotnode(0,5){X4}
\ncline{X1}{X2}
\ncline{X2}{X3}
\ncline{X3}{X4}
\ncline{X4}{X1}

\dotnode(1.5,1.5){Y1} 
\dotnode(3.5,1.5){Y2} 
\dotnode(3.5,3.5){Y3}
\dotnode(1.5,3.5){Y4}
\ncline{Y1}{Y2}
\ncline{Y2}{Y3}
\ncline{Y3}{Y4}
\ncline{Y4}{Y1}

\ncline{X1}{Y1}
\ncline{X2}{Y2}
\ncline{X3}{Y3}
\ncline{X4}{Y4}

\put(-0.6,0.1){$x'_1$}
\put(5.1,0.1){$x'_2$}
\put(-0.6,4.7){$x'_4$}
\put(5.1,4.7){$x'_3$}

\put(1.6,1.6){$y'_1$}
\put(3.6,1.6){$y'_2$}
\put(3.6,3.2){$y'_3$}
\put(1.6,3.2){$y'_4$}

\end{pspicture}}

\pscurve(0.5,5.5)(4.5,6)(8.5,5.5)
\pscurve(5.5,5.5)(9,6)(13.5,5.5)
\pscurve(0.5,0.5)(4.5,0)(8.5,0.5)
\pscurve(5.5,0.5)(9,0)(13.5,0.5)

\pscurve(2,4)(6,3.4)(10,4)
\pscurve(4,4)(8,4.5)(12,4)
\pscurve(2,2)(6,2.6)(10,2)
\pscurve(4,2)(8,1.5)(12,2)

\rput(9,0){\circlenode[linestyle=none,fillstyle=solid,fillcolor=white]{C1}{\mbox{\small 4}}}
\rput(7,1.5){\circlenode[linestyle=none,fillstyle=solid,fillcolor=white]{C1}{\mbox{\small 4}}}

\end{pspicture}
\end{center}

\caption {Subcase 1B} 
\label{prop4_case1A2}
\end{figure}

$\newline$
\noindent{\textbf{Case 2}} There are two edges of the same dimension and the same color at distance two.

Without loss of generality, we assume that $f(x_1 x_2)=f(y_3 y_4)=1$. Moreover, we assume that $f(x_3 y_3)=3$ and $f(y_2 y_3)=2$. Then, because of R1 and Case 1, the edge $x_2 y_2$ is colored $4$. However, due to R1 and Case 1, the edge $x_2 x_3$ cannot be colored with any of the four colors, a contradiction. 

$\newline$
\noindent{\textbf{Case 3}} There are two edges of the same dimension and the same color at distance 3.

From previous cases, for each dimension, there are four pairs of edges at distance 3 with the same color. Without loss of generality, we assume that $f(x_1 x_2)=f(y'_3 y'_4)=1$. Additionally, we assume that $f(x_2 x_3)=f(y'_1y'_4)=2$. The color of the edges $x_2y_2$ and $x'_4y'_4$ is the same in $f$. Hence, the vertices $x_2$ and $y'_4$ receive the same sequence, a contradiction.
\end{proof}

\begin{ob}
Five colors are enough to distinguish the vertices of a hypercube $H_4$ by proper edge coloring.
\end{ob}
\begin{proof}
The observation results from Figure \ref{H_4}.
\end{proof}

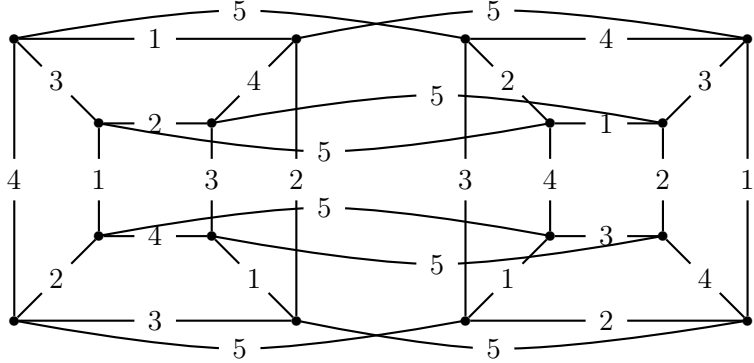
\begin{figure}
\psset{unit=0.75cm}
\psset{radius=0.2}

\begin{center}

\begin{pspicture}(13,6)

\put(0.5,0.5)
{\begin{pspicture}(0,0)
\dotnode(0,0){X1} 
\dotnode(5,0){X2} 
\dotnode(5,5){X3}
\dotnode(0,5){X4}
\ncline{X1}{X2}
\ncline{X2}{X3}
\ncline{X3}{X4}
\ncline{X4}{X1}

\dotnode(1.5,1.5){Y1} 
\dotnode(3.5,1.5){Y2} 
\dotnode(3.5,3.5){Y3}
\dotnode(1.5,3.5){Y4}
\ncline{Y1}{Y2}
\ncline{Y2}{Y3}
\ncline{Y3}{Y4}
\ncline{Y4}{Y1}

\ncline{X1}{Y1}
\ncline{X2}{Y2}
\ncline{X3}{Y3}
\ncline{X4}{Y4}

\rput(2.5,0){\circlenode[linestyle=none,fillstyle=solid,fillcolor=white]{C1}{\mbox{\small 3}}}
\rput(2.5,5){\circlenode[linestyle=none,fillstyle=solid,fillcolor=white]{C1}{\mbox{\small 1}}}
\rput(0.75,0.75){\circlenode[linestyle=none,fillstyle=solid,fillcolor=white]{C1}{\mbox{\small 2}}}
\rput(4.25,0.75){\circlenode[linestyle=none,fillstyle=solid,fillcolor=white]{C1}{\mbox{\small 1}}}
\rput(5,2.5){\circlenode[linestyle=none,fillstyle=solid,fillcolor=white]{C1}{\mbox{\small 2}}}
\rput(0,2.5){\circlenode[linestyle=none,fillstyle=solid,fillcolor=white]{C1}{\mbox{\small 4}}}
\rput(1.5,2.5){\circlenode[linestyle=none,fillstyle=solid,fillcolor=white]{C1}{\mbox{\small 1}}}

\rput(2.5,1.5){\circlenode[linestyle=none,fillstyle=solid,fillcolor=white]{C1}{\mbox{\small 4}}}
\rput(0.75,4.25){\circlenode[linestyle=none,fillstyle=solid,fillcolor=white]{C1}{\mbox{\small 3}}}
\rput(4.25,4.25){\circlenode[linestyle=none,fillstyle=solid,fillcolor=white]{C1}{\mbox{\small 4}}}
\rput(2.5,3.5){\circlenode[linestyle=none,fillstyle=solid,fillcolor=white]{C1}{\mbox{\small 2}}}
\rput(3.5,2.5){\circlenode[linestyle=none,fillstyle=solid,fillcolor=white]{C1}{\mbox{\small 3}}}

\end{pspicture}}

\put(8.5,0.5)
{\begin{pspicture}(0,0)
\dotnode(0,0){X1} 
\dotnode(5,0){X2} 
\dotnode(5,5){X3}
\dotnode(0,5){X4}
\ncline{X1}{X2}
\ncline{X2}{X3}
\ncline{X3}{X4}
\ncline{X4}{X1}

\dotnode(1.5,1.5){Y1} 
\dotnode(3.5,1.5){Y2} 
\dotnode(3.5,3.5){Y3}
\dotnode(1.5,3.5){Y4}
\ncline{Y1}{Y2}
\ncline{Y2}{Y3}
\ncline{Y3}{Y4}
\ncline{Y4}{Y1}

\ncline{X1}{Y1}
\ncline{X2}{Y2}
\ncline{X3}{Y3}
\ncline{X4}{Y4}

\rput(2.5,0){\circlenode[linestyle=none,fillstyle=solid,fillcolor=white]{C1}{\mbox{\small 2}}}
\rput(2.5,5){\circlenode[linestyle=none,fillstyle=solid,fillcolor=white]{C1}{\mbox{\small 4}}}
\rput(0.75,0.75){\circlenode[linestyle=none,fillstyle=solid,fillcolor=white]{C1}{\mbox{\small 1}}}
\rput(4.25,0.75){\circlenode[linestyle=none,fillstyle=solid,fillcolor=white]{C1}{\mbox{\small 4}}}
\rput(5,2.5){\circlenode[linestyle=none,fillstyle=solid,fillcolor=white]{C1}{\mbox{\small 1}}}
\rput(0,2.5){\circlenode[linestyle=none,fillstyle=solid,fillcolor=white]{C1}{\mbox{\small 3}}}
\rput(1.5,2.5){\circlenode[linestyle=none,fillstyle=solid,fillcolor=white]{C1}{\mbox{\small 4}}}

\rput(2.5,1.5){\circlenode[linestyle=none,fillstyle=solid,fillcolor=white]{C1}{\mbox{\small 3}}}
\rput(0.75,4.25){\circlenode[linestyle=none,fillstyle=solid,fillcolor=white]{C1}{\mbox{\small 2}}}
\rput(4.25,4.25){\circlenode[linestyle=none,fillstyle=solid,fillcolor=white]{C1}{\mbox{\small 3}}}
\rput(2.5,3.5){\circlenode[linestyle=none,fillstyle=solid,fillcolor=white]{C1}{\mbox{\small 1}}}
\rput(3.5,2.5){\circlenode[linestyle=none,fillstyle=solid,fillcolor=white]{C1}{\mbox{\small 2}}}

\end{pspicture}}

\pscurve(0.5,5.5)(4.5,6)(8.5,5.5)
\pscurve(5.5,5.5)(9,6)(13.5,5.5)
\pscurve(0.5,0.5)(4.5,0)(8.5,0.5)
\pscurve(5.5,0.5)(9,0)(13.5,0.5)

\pscurve(2,4)(6,3.5)(10,4)
\pscurve(4,4)(8,4.5)(12,4)
\pscurve(2,2)(6,2.5)(10,2)
\pscurve(4,2)(8,1.5)(12,2)

\rput(4.5,0){\circlenode[linestyle=none,fillstyle=solid,fillcolor=white]{C1}{\mbox{\small 5}}}
\rput(9,0){\circlenode[linestyle=none,fillstyle=solid,fillcolor=white]{C1}{\mbox{\small 5}}}
\rput(8,1.5){\circlenode[linestyle=none,fillstyle=solid,fillcolor=white]{C1}{\mbox{\small 5}}}  
\rput(4.5,6){\circlenode[linestyle=none,fillstyle=solid,fillcolor=white]{C1}{\mbox{\small 5}}}
\rput(9,6){\circlenode[linestyle=none,fillstyle=solid,fillcolor=white]{C1}{\mbox{\small 5}}}
\rput(6,3.5){\circlenode[linestyle=none,fillstyle=solid,fillcolor=white]{C1}{\mbox{\small 5}}}  
\rput(8,4.5){\circlenode[linestyle=none,fillstyle=solid,fillcolor=white]{C1}{\mbox{\small 5}}}
\rput(6,2.5){\circlenode[linestyle=none,fillstyle=solid,fillcolor=white]{C1}{\mbox{\small 5}}}

\end{pspicture}
\end{center}

\caption {Proper edge coloring of $H_4$ distinguishing all vertices} 
\label{H_4}
\end{figure}

\subsection{Cases for $n\geq 5$}

Finally, we show that for hypercubes of dimension at least five, the natural lower bound of five colors is sufficient.

\begin{thm} \label{thm:proper_n}
The smallest number of colors in a proper coloring needed to distinguish the vertices of a hypercube $H_n$, where $n\geq 5$, is $n$. 
\end{thm}
\begin{proof}
Let $f_n: E \to \{1,\dots,k\}$ be the proper edge coloring of the hypercube $H_n$ with the smallest possible numbers of colors, such that all vertices are distinguished. Since $f_n$ is a proper coloring and each vertex is incident to $n$ edges, the number of colors in the coloring $f_n$ is at least $n$. We show that there exists a desired coloring with $n$ colors.

The coloring $f_5$ of the hypercube $H_5$ with five colors is shown in Figure \ref{fig5}. Below are all palletes of the vertices. 

\begin{center}
45321 $\;$ 35421 $\;$ 54321 $\;$ 53421\newline
43521 $\;$ 35124 $\;$ 54231 $\;$ 52134\newline
24351 $\;$ 23451 $\;$ 42351 $\;$ 32451\newline
24531 $\;$ 25134 $\;$ 43251 $\;$ 32154\newline
54123 $\;$ 43125 $\;$ 52143 $\;$ 32145\newline
54132 $\;$ 45231 $\;$ 53142 $\;$ 32541\newline
25143 $\;$ 35142 $\;$ 24153 $\;$ 43152\newline
23145 $\;$ 35241 $\;$ 24135 $\;$ 42531\newline
\end{center}

Note that the edges of color 1 appear only in two dimensions (the 3rd and the 5th) and every vertex has an edge of color 1 incident to it.

For hypercubes $H_n$, $n\geq 6$, we will show inductively on $n$ that there is a coloring $f_n$ such that in one dimension (let us say the $n$th dimension) all edges have the same new color. 

To obtain the coloring $f_6$ of the hypercube $H_6$, we consider the hypercube $H_6$ as two five-dimensional hypercubes, $H_5$ and $H_5'$, with a matching connecting the corresponding vertices. Let us color the edges of the hypercube $H_5$ as shown in Figure \ref{fig5}. The coloring of the edges of $H'_5$ is obtained from the coloring of $H_5$ by applying the permutation $(1,2,3,4,5)$ to the dimensions. Consider an edge $(x_1,x_2,x_3,x_4,x_5)(y_1,y_2,y_3,y_4,y_5)$ of $H_5$, colored with $k \in \{1,2,3,4,5\}$. Then in $H_5'$, the edge $(x_5,x_1,x_2,x_3,x_4)(y_5,y_1,y_2,y_3,y_4)$ is colored with $k$. The resulting coloring of $H_5'$ has the property that the edges of color 1 appear only in dimensions 4 and 1. 

We color the remaining edges (of the 6th dimension) with color $6$. It is immediate that the coloring $f_6$ is a proper edge coloring. Moreover, if we consider a vertex from $H_5$ and a vertex from $H_5'$, their palettes are different. The vertex from $H_5$ has color 1 in the 3rd or the 5th position in the sequence, and the vertex in $H_5'$ has color 1 in the 4th or the 1st position. Hence, the coloring $f_6$ distinguishes all vertices of the graph $H_6$. This establishes the base case.

Suppose that $n\geq 7$ and there is a coloring $f_{n-1}$ of the hypercube $H_{n-1}$ such that in $(n-1)$st dimension all edges have the same new color (say $n-1$). Again, we will consider the graph $H_n$ as two hypercubes $H_{n-1}$ and $H_{n-1}'$ with a matching added. Now we define the coloring $f_n$ as follows: let us color the edges of the hypercubes $H_{n-1}$ and $H_{n-1}'$ as in the coloring $f_{n-1}$, and next interchange colors $n-1$ and $n-2$ in the hypercube $H_{n-1}'$. We color the remaining edges with the new color $n$. The claim is a consequence of the inductive hypothesis.
\end{proof}

\begin{figure} 
\psset{unit=0.75cm}
\psset{radius=0.2}

\begin{center}

\begin{pspicture}(13,10)

\put(0.5,0.5)
{\begin{pspicture}(0,0)
\dotnode(0,0){X1} 
\dotnode(5,0){X2} 
\dotnode(5,5){X3}
\dotnode(0,5){X4}
\ncline{X1}{X2}
\ncline{X2}{X3}
\ncline{X3}{X4}
\ncline{X4}{X1}

\dotnode(1.5,1.5){Y1} 
\dotnode(3.5,1.5){Y2} 
\dotnode(3.5,3.5){Y3}
\dotnode(1.5,3.5){Y4}
\ncline{Y1}{Y2}
\ncline{Y2}{Y3}
\ncline{Y3}{Y4}
\ncline{Y4}{Y1}

\ncline{X1}{Y1}
\ncline{X2}{Y2}
\ncline{X3}{Y3}
\ncline{X4}{Y4}


\rput(2.5,0){\circlenode[linestyle=none,fillstyle=solid,fillcolor=white]{C1}{\mbox{\small 2}}}
\rput(2.5,5){\circlenode[linestyle=none,fillstyle=solid,fillcolor=white]{C1}{\mbox{\small 5}}}
\rput(0.75,0.75){\circlenode[linestyle=none,fillstyle=solid,fillcolor=white]{C1}{\mbox{\small 1}}}
\rput(4.25,0.75){\circlenode[linestyle=none,fillstyle=solid,fillcolor=white]{C1}{\mbox{\small 1}}}
\rput(5,2.5){\circlenode[linestyle=none,fillstyle=solid,fillcolor=white]{C1}{\mbox{\small 4}}}
\rput(0,2.5){\circlenode[linestyle=none,fillstyle=solid,fillcolor=white]{C1}{\mbox{\small 4}}}
\rput(1.5,2.5){\circlenode[linestyle=none,fillstyle=solid,fillcolor=white]{C1}{\mbox{\small 2}}}

\rput(2.5,1.5){\circlenode[linestyle=none,fillstyle=solid,fillcolor=white]{C1}{\mbox{\small 5}}}
\rput(0.75,4.25){\circlenode[linestyle=none,fillstyle=solid,fillcolor=white]{C1}{\mbox{\small 3}}}
\rput(4.25,4.25){\circlenode[linestyle=none,fillstyle=solid,fillcolor=white]{C1}{\mbox{\small 2}}}
\rput(2.5,3.5){\circlenode[linestyle=none,fillstyle=solid,fillcolor=white]{C1}{\mbox{\small 4}}}
\rput(3.5,2.5){\circlenode[linestyle=none,fillstyle=solid,fillcolor=white]{C1}{\mbox{\small 3}}}

\end{pspicture}}

\put(8.5,0.5)
{\begin{pspicture}(0,0)
\dotnode(0,0){X1} 
\dotnode(5,0){X2} 
\dotnode(5,5){X3}
\dotnode(0,5){X4}
\ncline{X1}{X2}
\ncline{X2}{X3}
\ncline{X3}{X4}
\ncline{X4}{X1}

\dotnode(1.5,1.5){Y1} 
\dotnode(3.5,1.5){Y2} 
\dotnode(3.5,3.5){Y3}
\dotnode(1.5,3.5){Y4}
\ncline{Y1}{Y2}
\ncline{Y2}{Y3}
\ncline{Y3}{Y4}
\ncline{Y4}{Y1}

\ncline{X1}{Y1}
\ncline{X2}{Y2}
\ncline{X3}{Y3}
\ncline{X4}{Y4}


\rput(2.5,0){\circlenode[linestyle=none,fillstyle=solid,fillcolor=white]{C1}{\mbox{\small 4}}}
\rput(2.5,5){\circlenode[linestyle=none,fillstyle=solid,fillcolor=white]{C1}{\mbox{\small 5}}}
\rput(0.75,0.75){\circlenode[linestyle=none,fillstyle=solid,fillcolor=white]{C1}{\mbox{\small 1}}}
\rput(4.25,0.75){\circlenode[linestyle=none,fillstyle=solid,fillcolor=white]{C1}{\mbox{\small 5}}}
\rput(5,2.5){\circlenode[linestyle=none,fillstyle=solid,fillcolor=white]{C1}{\mbox{\small 2}}}
\rput(0,2.5){\circlenode[linestyle=none,fillstyle=solid,fillcolor=white]{C1}{\mbox{\small 3}}}
\rput(1.5,2.5){\circlenode[linestyle=none,fillstyle=solid,fillcolor=white]{C1}{\mbox{\small 2}}}

\rput(2.5,1.5){\circlenode[linestyle=none,fillstyle=solid,fillcolor=white]{C1}{\mbox{\small 3}}}
\rput(0.75,4.25){\circlenode[linestyle=none,fillstyle=solid,fillcolor=white]{C1}{\mbox{\small 4}}}
\rput(4.25,4.25){\circlenode[linestyle=none,fillstyle=solid,fillcolor=white]{C1}{\mbox{\small 1}}}
\rput(2.5,3.5){\circlenode[linestyle=none,fillstyle=solid,fillcolor=white]{C1}{\mbox{\small 3}}}
\rput(3.5,2.5){\circlenode[linestyle=none,fillstyle=solid,fillcolor=white]{C1}{\mbox{\small 2}}}

\end{pspicture}}
\pscurve(0.5,5.5)(4.5,6)(8.5,5.5)
\pscurve(5.5,5.5)(9,6)(13.5,5.5)
\pscurve(0.5,0.5)(4.5,0)(8.5,0.5)
\pscurve(5.5,0.5)(9,0)(13.5,0.5)

\pscurve(2,4)(6,3.5)(10,4)
\pscurve(4,4)(8,4.5)(12,4)
\pscurve(2,2)(6,2.5)(10,2)
\pscurve(4,2)(8,1.5)(12,2)

\rput(5,0){\circlenode[linestyle=none, fillstyle=solid,fillcolor=white]{C1}{\mbox{\small 5}}}
\rput(9,0){\circlenode[linestyle=none,fillstyle=solid,fillcolor=white]{C1}{\mbox{\small 3}}}
\rput(7.5,1.5){\circlenode[linestyle=none,fillstyle=solid,fillcolor=white]{C1}{\mbox{\small 4}}}  
\rput(5,6){\circlenode[linestyle=none,fillstyle=solid,fillcolor=white]{C1}{\mbox{\small 2}}}
\rput(9,6){\circlenode[linestyle=none,fillstyle=solid,fillcolor=white]{C1}{\mbox{\small 3}}}
\rput(6.5,3.5){\circlenode[linestyle=none,fillstyle=solid,fillcolor=white]{C1}{\mbox{\small 5}}} 
\rput(7.5,4.5){\circlenode[linestyle=none,fillstyle=solid,fillcolor=white]{C1}{\mbox{\small 5}}}
\rput(6.5,2.5){\circlenode[linestyle=none,fillstyle=solid,fillcolor=white]{C1}{\mbox{\small 4}}}


\put(0.5, 8.5)
{\begin{pspicture}(0,0)
\dotnode(0,0){X1} 
\dotnode(5,0){X2} 
\dotnode(5,5){X3}
\dotnode(0,5){X4}
\ncline{X1}{X2}
\ncline{X2}{X3}
\ncline{X3}{X4}
\ncline{X4}{X1}

\dotnode(1.5,1.5){Y1} 
\dotnode(3.5,1.5){Y2} 
\dotnode(3.5,3.5){Y3}
\dotnode(1.5,3.5){Y4}
\ncline{Y1}{Y2}
\ncline{Y2}{Y3}
\ncline{Y3}{Y4}
\ncline{Y4}{Y1}

\ncline{X1}{Y1}
\ncline{X2}{Y2}
\ncline{X3}{Y3}
\ncline{X4}{Y4}

\rput(2.5,0){\circlenode[linestyle=none,fillstyle=solid,fillcolor=white]{C1}{\mbox{\small 2}}}
\rput(2.5,5){\circlenode[linestyle=none,fillstyle=solid,fillcolor=white]{C1}{\mbox{\small 4}}}
\rput(0.75,0.75){\circlenode[linestyle=none,fillstyle=solid,fillcolor=white]{C1}{\mbox{\small 1}}}
\rput(4.25,0.75){\circlenode[linestyle=none,fillstyle=solid,fillcolor=white]{C1}{\mbox{\small 1}}}
\rput(5,2.5){\circlenode[linestyle=none,fillstyle=solid,fillcolor=white]{C1}{\mbox{\small 3}}}
\rput(0,2.5){\circlenode[linestyle=none,fillstyle=solid,fillcolor=white]{C1}{\mbox{\small 5}}}
\rput(1.5,2.5){\circlenode[linestyle=none,fillstyle=solid,fillcolor=white]{C1}{\mbox{\small 4}}}

\rput(2.5,1.5){\circlenode[linestyle=none,fillstyle=solid,fillcolor=white]{C1}{\mbox{\small 5}}}
\rput(0.75,4.25){\circlenode[linestyle=none,fillstyle=solid,fillcolor=white]{C1}{\mbox{\small 3}}}
\rput(4.25,4.25){\circlenode[linestyle=none,fillstyle=solid,fillcolor=white]{C1}{\mbox{\small 5}}}
\rput(2.5,3.5){\circlenode[linestyle=none,fillstyle=solid,fillcolor=white]{C1}{\mbox{\small 2}}}
\rput(3.5,2.5){\circlenode[linestyle=none,fillstyle=solid,fillcolor=white]{C1}{\mbox{\small 4}}}

\end{pspicture}}

\put(8.5,8.5)
{\begin{pspicture}(0,0)
\dotnode(0,0){X1} 
\dotnode(5,0){X2} 
\dotnode(5,5){X3}
\dotnode(0,5){X4}
\ncline{X1}{X2}
\ncline{X2}{X3}
\ncline{X3}{X4}
\ncline{X4}{X1}

\dotnode(1.5,1.5){Y1} 
\dotnode(3.5,1.5){Y2} 
\dotnode(3.5,3.5){Y3}
\dotnode(1.5,3.5){Y4}
\ncline{Y1}{Y2}
\ncline{Y2}{Y3}
\ncline{Y3}{Y4}
\ncline{Y4}{Y1}

\ncline{X1}{Y1}
\ncline{X2}{Y2}
\ncline{X3}{Y3}
\ncline{X4}{Y4}

\rput(2.5,0){\circlenode[linestyle=none,fillstyle=solid,fillcolor=white]{C1}{\mbox{\small 3}}}
\rput(4.25,0.75){\circlenode[linestyle=none,fillstyle=solid,fillcolor=white]{C1}{\mbox{\small 2}}}
\rput(5,2.5){\circlenode[linestyle=none,fillstyle=solid,fillcolor=white]{C1}{\mbox{\small 5}}}
\rput(2.5,5){\circlenode[linestyle=none,fillstyle=solid,fillcolor=white]{C1}{\mbox{\small 3}}}
\rput(0.75,0.75){\circlenode[linestyle=none,fillstyle=solid,fillcolor=white]{C1}{\mbox{\small 1}}}
\rput(0,2.5){\circlenode[linestyle=none,fillstyle=solid,fillcolor=white]{C1}{\mbox{\small 5}}}
\rput(1.5,2.5){\circlenode[linestyle=none,fillstyle=solid,fillcolor=white]{C1}{\mbox{\small 3}}}
\rput(2.5,1.5){\circlenode[linestyle=none,fillstyle=solid,fillcolor=white]{C1}{\mbox{\small 4}}}
\rput(0.75,4.25){\circlenode[linestyle=none,fillstyle=solid,fillcolor=white]{C1}{\mbox{\small 4}}}
\rput(4.25,4.25){\circlenode[linestyle=none,fillstyle=solid,fillcolor=white]{C1}{\mbox{\small 1}}}
\rput(2.5,3.5){\circlenode[linestyle=none,fillstyle=solid,fillcolor=white]{C1}{\mbox{\small 2}}}
\rput(3.5,2.5){\circlenode[linestyle=none,fillstyle=solid,fillcolor=white]{C1}{\mbox{\small 5}}}

\end{pspicture}}
\pscurve(0.5,13.5)(4.5,14)(8.5,13.5)
\pscurve(5.5,13.5)(9,14)(13.5,13.5)
\pscurve(0.5,8.5)(4.5,8)(8.5,8.5)
\pscurve(5.5,8.5)(9,8)(13.5,8.5)

\pscurve(2,12)(6,11.5)(10,12)
\pscurve(4,12)(8,12.5)(12,12)
\pscurve(2,10)(6,10.5)(10,10)
\pscurve(4,10)(8,9.5)(12,10)

\rput(5,14){\circlenode[linestyle=none,fillstyle=solid,fillcolor=white]{C1}{\mbox{\small 2}}}
\rput(9,14){\circlenode[linestyle=none,fillstyle=solid,fillcolor=white]{C1}{\mbox{\small 2}}}
\rput(5,8){\circlenode[linestyle=none,fillstyle=solid,fillcolor=white]{C1}{\mbox{\small 4}}}  
\rput(9,8){\circlenode[linestyle=none,fillstyle=solid,fillcolor=white]{C1}{\mbox{\small 4}}}
\rput(6.5,11.5){\circlenode[linestyle=none,fillstyle=solid,fillcolor=white]{C1}{\mbox{\small 5}}}
\rput(7.5,12.5){\circlenode[linestyle=none,fillstyle=solid,fillcolor=white]{C1}{\mbox{\small 3}}} 
\rput(6.5,10.5){\circlenode[linestyle=none,fillstyle=solid,fillcolor=white]{C1}{\mbox{\small 2}}}
\rput(7.5,9.5){\circlenode[linestyle=none,fillstyle=solid,fillcolor=white]{C1}{\mbox{\small 3}}}

\pscurve(0.5,13.5)(0,9)(0.5,5.5)
\pscurve(5.5,13.5)(6,9)(5.5,5.5)
\pscurve(0.5,8.5)(0,5)(0.5,0.5)
\pscurve(5.5,8.5)(6,5)(5.5,0.5)

\pscurve(2,12)(1.5,7.5)(2,4)
\pscurve(4,12)(3.5,7.5)(4,4)
\pscurve(2,10)(2.5,5.5)(2,2)
\pscurve(4,10)(4.5,5.5)(4,2)

\pscurve(8.5,13.5)(8,9)(8.5,5.5)
\pscurve(13.5,13.5)(14,9)(13.5,5.5)
\pscurve(8.5,8.5)(8,5)(8.5,0.5)
\pscurve(13.5,8.5)(14,5)(13.5,0.5)

\pscurve(10,12)(9.5,7.5)(10,4)
\pscurve(12,12)(11.5,7.5)(12,4)
\pscurve(10,10)(10.5,5.5)(10,2)
\pscurve(12,10)(12.5,5.5)(12,2)

\rput(0,9){\circlenode[linestyle=none,fillstyle=solid,fillcolor=white]{C1}{\mbox{\small 1}}}
\rput(6,9){\circlenode[linestyle=none,fillstyle=solid,fillcolor=white]{C1}{\mbox{\small 1}}}
\rput(0,5){\circlenode[linestyle=none,fillstyle=solid,fillcolor=white]{C1}{\mbox{\small 3}}}  
\rput(6,5){\circlenode[linestyle=none,fillstyle=solid,fillcolor=white]{C1}{\mbox{\small 5}}}
\rput(1.5,7.5){\circlenode[linestyle=none,fillstyle=solid,fillcolor=white]{C1}{\mbox{\small 1}}}
\rput(3.5,7.5){\circlenode[linestyle=none,fillstyle=solid,fillcolor=white]{C1}{\mbox{\small 1}}} 
\rput(2.5,7){\circlenode[linestyle=none,fillstyle=solid,fillcolor=white]{C1}{\mbox{\small 3}}}
\rput(4.5,7){\circlenode[linestyle=none,fillstyle=solid,fillcolor=white]{C1}{\mbox{\small 2}}}
\rput(8,9){\circlenode[linestyle=none,fillstyle=solid,fillcolor=white]{C1}{\mbox{\small 1}}}
\rput(14,9){\circlenode[linestyle=none,fillstyle=solid,fillcolor=white]{C1}{\mbox{\small 4}}}
\rput(8,5){\circlenode[linestyle=none,fillstyle=solid,fillcolor=white]{C1}{\mbox{\small 2}}}  
\rput(14,5){\circlenode[linestyle=none,fillstyle=solid,fillcolor=white]{C1}{\mbox{\small 1}}}
\rput(9.5,7.5){\circlenode[linestyle=none,fillstyle=solid,fillcolor=white]{C1}{\mbox{\small 1}}}
\rput(11.5,7.5){\circlenode[linestyle=none,fillstyle=solid,fillcolor=white]{C1}{\mbox{\small 4}}} 
\rput(10.5,7){\circlenode[linestyle=none,fillstyle=solid,fillcolor=white]{C1}{\mbox{\small 5}}}
\rput(12.5,7){\circlenode[linestyle=none,fillstyle=solid,fillcolor=white]{C1}{\mbox{\small 1}}}

\end{pspicture}
\end{center}

\caption {Proper 5-edge coloring of $H_5$ distinguishing all vertices}
\label{fig5} 

\end{figure}
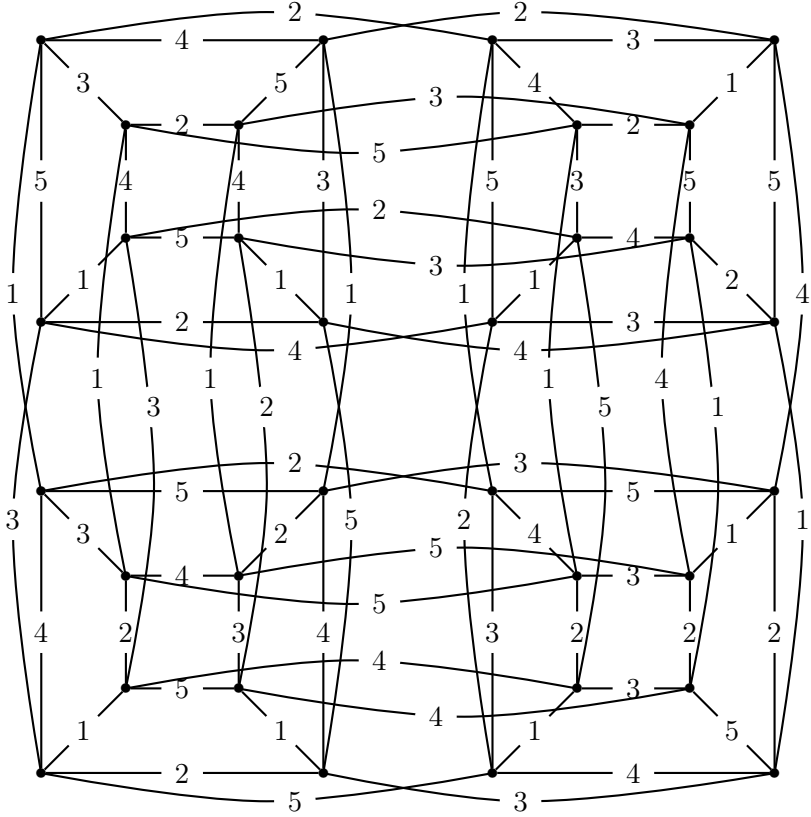
\section{Final remarks}

As we mentioned in the introduction, the idea of considering sequences in the context of distinguishing vertices of a graph appeared in \cite{SeSt13}.
The main focus of \cite{SeSt13} is the coloring that distinguishes adjacent vertices. However, authors also consider distinguishing all vertices. We note that the bounds given in Theorem \ref{thm:general_2} and Theorem \ref{thm:proper_n} are best possible when we only require for adjacent vertices to receive different sequences.

In their approach, the ordering of the edges incident with a given vertex always follows from the ordering of the entire edge set. Since in their paper the ordering of the edges can be introduced arbitrarily, they consider two parameters.
The first of them, called {\it general sequence irregularity strength}, is the smallest number $k$ of colors such that for every ordering of $E(G)$ there exists an edge coloring of $G$ with $k$ colors, resulting in every vertex receiving a distinct sequence of colors.
The second one, called {\it specific sequence irregularity strength}, is defined similarly with the 
 requirement that {\bf there exists} an ordering of the edge set for which there is a coloring distinguishing all vertices by sequences.

Denoting by $n_i$ the number of vertices of degree $i$ in a graph $G$, and by $k$ the number of colors in an edge coloring distinguishing all vertices by sequences, it is immediate that $k^{i} \ge n_i$  for $1 \le  i \le \Delta (G)$. 
Now let $M_G =\max \{ \left\lceil \sqrt[i]{n_i} \right\rceil \colon 1 \le  i \le \Delta (G)\}$. The authors in \cite{SeSt13} conjectured that general sequence irregularity strength of $G$ is always equal to $M_G$.
However, such a generally formulated conjecture is not true. Indeed, let us notice that for the graph $H_2$ and the ordering
$e_1 = ({00, 01}), e_2 = ({00, 10}), e_3 = ({10, 11}), e_4 = ({01, 11})$ there does not exist an edge coloring with two colors such that every vertex receives a distinct sequence.

However, the following, slightly less general conjecture may be true.

\begin{conj} 
If $G$ is a graph without a component isomorphic to $K_2$, then specific sequence irregularity strength of $G$ is equal to $M_G$.
\end{conj}

In our approach, we consider a natural ordering of the edges incident with a given vertex. This differs from the approach considered in \cite{SeSt13}, as we consider a local ordering opposed to the global ordering of the edges. However, we note that the ordering of the edges that we consider can also be thought of as coming from a general ordering of the edges. Just number the edges of dimension $1$ with numbers from $1$ to $n$, the edges of dimension $2$ with numbers from $n+1$ to $2n$, and so on. Therefore, Theorem 1 provides support for the above conjecture, confirming it for the hypercubes.

As far as we know, the sequence variant of vertex distinguishing has not been considered so far in the case of proper edge colorings.


\end{document}